\documentclass[UTF8]{article}

\usepackage{graphicx}
\usepackage{hyperref}

\usepackage{amssymb}
\usepackage[thmmarks]{ntheorem}
\usepackage[tbtags]{amsmath}

\usepackage{caption}
\usepackage{subcaption}

\usepackage{tikz}       
\usetikzlibrary{calc,arrows}

\usepackage[bottom=25mm,left=25mm,right=25mm,top=25mm]{geometry}

{   
	\theoremstyle{nonumberplain}
	\theoremheaderfont{\bfseries}
	\theorembodyfont{\normalfont}
	\theoremsymbol{$\square$}
	\newtheorem{proof}{Proof}
}

\allowdisplaybreaks[3]

\newtheorem{definition}{Definition}[section]
\newtheorem{theorem}{Theorem}[section]
\newtheorem{corollary}[theorem]{Corollary}
\newtheorem{lemma}[theorem]{Lemma}
\newtheorem{proposition}[theorem]{Proposition}
\newtheorem{remark}[theorem]{Remark}

\begin{document}

\title{Convex expansion for finite distributive lattices with applications\thanks{This work was supported by NSFC (Grant No. 11761064).}
}

\author{Xu Wang, Xuxu Zhao and Haiyuan Yao\footnote{Corresponding author.}%
	\\ {\footnotesize College of Mathematics and Statistics, Northwest Normal University, Lanzhou 730070, PR China}}
\date{}

\maketitle

\begin{abstract}
The concept of cutting is first explicitly introduced. By the concept, a convex expansion for finite distributive lattices is considered. Thus, a more general method for drawing the Hasse diagram is given, and the rank generating function of a finite distributive lattice is obtained.
In addition, we have several enumerative properties on finite distributive lattices and verify the generalized Euler formula for polyhedrons.


\textbf{2010 AMS Subj. Class.:} 06D05; 06A07; 06B05 

\end{abstract}

\section{Introduction}

Day introduced a doubling construction in a proof in \cite{aDay1970}. Using the construction, J\'onsson \cite{aJonss1971} considered the free products of lattices, Day \cite{aDay1977} provided an affirmative answer to a McKenzie's problem \cite{aMcKen1972}, Davey et al.\ \cite{aDaveyDQR1978} studied the exponents of lattices for finite distributive lattices, and Siv\'ak \cite{aSivak1978} characterized a class of finite lattices put by Schein \cite{aSchei1972}. Day et al.\ \cite{aDayGP1979} obtained a result of the construction on distributive lattice. In addition, Day et al.\ \cite{aDayNT1989} found various applications of doubling construction in the study of finite lattices etc., and Day \cite{aDay1992} obtained some results for the construction. Geyer \cite{aGeyer1993} extended the construction to obtain a lattice, and Nation \cite{aNatio1995} surveyed the doubling construction. Recently, Ern\'{e} et al.\ \cite{aErneHR2002} obtained a convex expansion for finite distributive lattices by the construction. Bertet et al.\ \cite{aBerteC2002} characterized the doubling convex sets in lattices and got a recognition algorithms.

In the paper a cutting of a finite distributive lattice is first explicitly introduced; in fact, Day \cite{aDay1992,aDayGP1979} already had the thought on cutting and got some results, we give two equivalent characterizations by the fundamental theorem for finite distributive lattices \cite{aBirkh1937}.
Using the concept of cutting, we obtain a convex expansion for finite distributive lattices; moreover, the expansion extends a result for finite distributive lattices in \cite{aErneHR2002} and the method for drawing the Hasse diagram of a filter lattice in \cite[P293]{bStanl2011}.
Using the convex expansion for finite distributive lattices, we have a series of consequences, including the rank generating function \cite{bStanl2011}, the number of convex Boolean sublattices of finite distributive lattice, and the number of particular elements; especially, the conclusions of Fibonacci and Lucas cubes \cite{aHsuPL1993,aMunarCS2001}. Finally, we verify the generalized Euler formula for polyhedrons by the relation of convex Boolean sublattices and antichains of posets.

\section{Preliminaries}

A set $P$ equipped with a partial order relation $\le$ is said to be a \emph{partially ordered set} (poset for short).
We write $x\parallel y$ if $x\nleq y$ and $y\ngeq x$.
A subposet $S$ of $P$ is a \emph{chain} if any two elements of $S$ are comparable, and denoted by $\mathbf{n}$ if $|S|=n$ \cite{bDaveyP2002}. The chain $S$ of $P$ is called \emph{saturated} if there does not exist $u \in P \setminus S$ such that $s < u < t$ for some $s, t \in S$ and such that $S\cup \{u\}$ is a chain \cite{bStanl2011}.
The set consisting of all minimal (resp.\ maximal) elements of $P$ is denoted by $\operatorname{Min}P$ (resp.\ $\operatorname{Max}P$).
Let $x\prec y$ denote $y$ \emph{cover}s $x$ in $P$, i.e.\ $x<y$ and $x\le z < y$ implies $z=x$.
The subset $Q$ of the poset $P$ is called \emph{convex} if $a,b \in Q$, $c \in P$, and $a \le c \le b$ imply that $c \in Q$.
The set of all filters of a poset $P$ is denoted by $\mathcal{F}(P)$, and carries the usual anti-inclusion order, forms a finite distributive lattice called \emph{filter lattice}.
Note that if $P=\emptyset$, then $\mathcal{F}(P) = \mathbf{1}$.
For a finite lattice $L$, we denote by $\hat0_L$ (resp. $\hat1_L$) the minimum (resp. maximum) element in $L$.
For a finite distributive lattice $L$, let $\mathop{\mathrm{Mi}}(L)$ denote the set of all meet-irreducible elements, regarded as a poset under the ordering of $L$ \cite{bGraet2011}.

For some concepts and notations not explained in the paper, refer to \cite{bDaveyP2002,bGraet2011,bStanl2011}.

\begin{theorem}[\cite{aBirkh1937,bStanl2011}]\label{th:ftfdl}
	Let $L$ be a finite distributive lattice. Then there is a unique (up to isomorphism) poset $P$ for which $L\cong \mathcal{F}(P)$. In fact $P\cong \mathop{\mathrm{Mi}}(L)$.
\end{theorem}

\section{Convex expansion}

\begin{definition}\label{def:cutsub}
	Let $L$ be a finite distributive lattice and let $K$ be a convex sublattice, i.e.\ interval, of $L$. We say the lattice $K$ is a \emph{cutting} of $L$ if any maximal chain of $L$ contains at least one element of $K$.
\end{definition}

Especially, if $K=\{a\}$ ($a \ne \hat0_L$ or $\hat1_L$) is a cutting of a finite distributive lattice $L$, then $a$ is called a \emph{cutting element} of $L$ \cite{aYaoZ2015}.

Let $L=\mathcal{F}(P)$ be a finite distributive lattice and let $K$ be a interval of $L$, then $S = \hat{0}_K\setminus \hat1_K$ is a convex subposet of $P$.
In addition, let $S_0:= \mathop{\mathrm{Max}}(P\setminus\hat{0}_K) = \{\,z \in P\setminus\hat{0}_K \mid \hat{0}_K\cup\{z\} \in\mathcal{F}(P)\,\}$, and $S_1:= \mathop{\mathrm{Min}}\hat{1}_K = \{\,y \in \hat{1}_K \mid \hat{1}_K\setminus\{y\} \in\mathcal{F}(P)\,\}$.
Clearly $P={\downarrow S_0} \mathbin{\dot\cup} S \mathbin{\dot\cup} {\uparrow S_1}$ and $z \not> y$ for all $z \in S_0,\,y \in S_1$, where $P\mathbin{\dot\cup}S$ is the disjoint union of two posets $P$ and $S$ \cite{bDaveyP2002}.

Let $K$ be a cutting of a finite distributive lattice $L = \mathcal{F}(P)$.
The poset $P_K:=(P\cup \{x_K\}, \le)$ is defined as follows: $x\le y$ if $x\le y$ in $P$ for all $x,y\in P$; and $x_K\succ z$ for all $z\in S_0$, $x_K\prec y$ for all $y\in S_1$ and $x_K\parallel w$ for all $w\in S$.
This partially order is well-defined by Theorem~\ref{th:cslt}.

Let $P$ be a poset and $x$ is an element not necessarily in $P$. We denote by $P-x$ and $P*x$ the induced subposet obtained from $P\setminus\{x\}$ and $P\setminus({\mathop{\uparrow} x}\cup{\mathop{\downarrow} x})$, respectively. Obviously $x\parallel w$ for every $w\in P*x$, and $P*x \subseteq (P-x)$ for all $x\in P$, moreover $P*x \cong P-x$ if and only if $P \cong (P-x) \mathbin{\dot{\cup}} \{x\}$.

\begin{theorem}\label{th:cslt}
	Let $K$ be a interval of a finite distributive lattice $L=\mathcal{F}(P)$. The posets $S$, $S_0$ and $S_1$ are defined as above, then
	the following are equivalent:
	\begin{enumerate}
		\itemsep=0em \parskip=0em
		\item[(1)] $L={\uparrow \hat{0}_K}\cup{\downarrow \hat{1}_K}$;
		\item[(2)] $K$ is a cutting of $L$;
		\item[(3)] $z<y$ for all $z\in S_0$ and $y\in S_1$;
		\item[(4)] there exist $x_K$ such that $S = P_K*x_K$.
	\end{enumerate}
\end{theorem}
\begin{proof}
	It is trivial that (3) $\iff$ (4) and it is not hard to verify (1) $\iff$ (2), we prove (2) $\iff$ (3) in what follows.
	
	Firstly, by contradiction, suppose that $z\in S_0$ (thus $z\notin \hat0_k$) and $y \in S_1 \subseteq \hat{1}_K$ such that $z \parallel y$, thus there exists $M\in K(\cong \mathcal{F}(S))$ such that $M\cup\{z\} \prec M \prec M\setminus\{y\}$, therefore $M\cup \{z\}\setminus \{y\}, \hat{0}_K\cup\{z\}, \hat{1}_K\setminus\{y\} \in \mathcal{F}(P)\setminus K$, it follows that there exists a saturated chain $\{\hat{0}_K\cup\{z\}, \dots, \hat{0}_K\cup\{z\}, M\cup \{z\}\setminus \{y\}, \hat{1}_K\setminus\{y\}, \dots, \hat{1}_K\setminus\{y\}\}$ containing no any element in $K$ (see Figure~\ref{fig:csl}), a contradiction.
	Therefore, the partial order on $P_K$ is well-defined.
	
	Conversely, let $P = \hat0_L \prec \cdots \prec M_0 \prec M \prec \cdots \prec \hat1_L=\emptyset$ be a saturated chain in $L$ such that $M_0\setminus M \subseteq S_0$ and $M \cap S_0 = \emptyset$. By the assumption, $\hat1_K = {\uparrow S_1} \subseteq M_0$, thus $\hat1_K \subseteq M$ and $\hat1_K \subseteq M \subseteq \hat0_K$, i.e.\ $M \in K$, thus $K$ is a cutting of $L$.
	%
\end{proof}

\begin{figure}[!htb] 
	\centering
	\begin{tikzpicture}[scale=0.8]
	\fill[lightgray] (-1,1) -- (0,2) -- (-1,3) -- (-2,2) -- cycle;
	
	\foreach \i in {0,1,2}
	{
		\foreach \k in {0,1}
		{
			\draw (\k-\i,\i+\k) -- (\k-\i+1,\i+\k+1)
			(\i-\k,\i+\k) -- (\i-\k-1,\i+1+\k);
		}
	}
	
	\foreach \i in {0,1,2}
	{
		\foreach \j in {0,1,2}
		{
			\filldraw[fill=white] (\i-\j,\i+\j) circle (1.5pt);
		}
	}
	
	\node at (-1,2) {$K$};
	\node[right] at (0,2) {$M$};
	\node[right] at (1,1) {$M\cup\{z\}$};
	\node[right] at (1,3) {$M\setminus\{y\}$};
	\node[right] at (2,2) {$M\cup\{z\}\setminus\{y\}$};
	\end{tikzpicture}
	\caption{Illustrating proof of Theorem~\ref{th:cslt}}\label{fig:csl}
\end{figure}
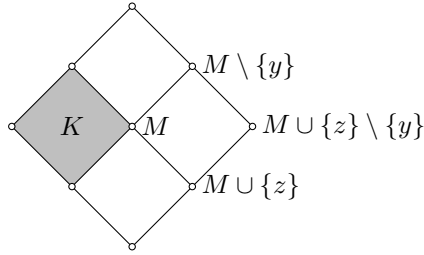

Observe that $\mathcal{F}(P_K)$ is a convex expansion of $L$, let $L\boxplus K$ denote the expansion, namely $L\boxplus K \cong \mathcal{F}(P_K)$ (see Figure~\ref{fig:lplk}, where $L\cong \mathcal{F}(P)$, $L\boxplus K \cong \mathcal{F}(P_K)$ and $K \cong \mathcal{F}(P*x_K)$); in general, if $L\boxplus K$ exists, then $K\boxplus L$ does not exist unless $L\cong K$.
Day \cite{aDay1992,aDayGP1979} had obtained a result on the distributivity of $L \boxplus K$.

\begin{lemma}[\cite{aDay1992,aDayGP1979}]
	Let $L$ be a finite distributive lattice. If $K \subseteq L$ is a interval, then $L \boxplus K$ ($L[K]$ in \cite{aDay1992,aDayGP1979}) is again distributive if and only if $L={\uparrow \hat{0}_K}\cup{\downarrow \hat{1}_K}$.
\end{lemma}

Moreover, it is easy to prove the following convex expansion for finite distributive lattices.

\begin{theorem}\label{th:cetfdl}
	If $K$ is a cutting of a finite distributive lattice $L = \mathcal{F}(P)$, then there exists $x_K$ such that
	\[
	\mathcal{F}(P_K) \cong \mathcal{F}(P) \boxplus \mathcal{F}(P_K*x_K).
	\]
	On the other hand, $\mathcal{F}(P)$ is also considered as a convex expansion, that is for every $x\in P$,
	\[
	\mathcal{F}(P)\cong \mathcal{F}(P-x)\boxplus \mathcal{F}(P*x).
	\]
\end{theorem}

\begin{figure}[!ht]  
	\centering
	\begin{subfigure}[b]{0.45\linewidth}
		\centering
		\begin{tikzpicture}
			\filldraw[fill=lightgray] (0,0) -- (2,1) -- (2,2) -- (0,1) -- cycle;
			\draw (0,1) -- (-1*0.306,1.732*0.306+1) -- (-1*0.306+2,1.732*0.306+2) -- (2,2)
			(0,0) -- (1*0.5,-1.732*0.5) -- (1*0.5+2,-1.732*0.5+1) -- (2,1);
			
			\foreach \i in {0,2}
			{
				\filldraw[fill=white] (\i,\i) circle (1.5pt);
			}
			\node at (1,1) {$K$};
			\node at (-0.2,0) {$\hat{0}_K$};
			\node at (2.25,2) {$\hat1_K$};
		\end{tikzpicture}
		\caption{the finite distributive lattice $L$ with a cutting $K$}
	\end{subfigure}
	\begin{subfigure}[b]{0.35\linewidth}
		\centering
		\begin{tikzpicture}
			\filldraw[fill=lightgray] (0,0) -- (2,1) -- (2,2) -- (0,1) -- cycle
			(0.707,-0.707) -- (2.707,1-0.707) -- (2.707,2-0.707) -- (0.707,1-0.707) -- cycle;
			\draw (0,1) -- (-1*0.306,1.732*0.306+1) -- (-1*0.306+2,1.732*0.306+2) -- (2,2);
			
			\draw[xshift=0.707cm,yshift=-0.707cm] (0,0) -- (1*0.5,-1.732*0.5) -- (1*0.5+2,-1.732*0.5+1) -- (2,1);
			
			\foreach \i in {0,1}
			{
				\foreach \j in {0,1}
				{
					\draw (2*\i,\i+\j)++(0.707,-0.707) -- (2*\i,\i+\j);
				}
			}
			\foreach \i in {0,2}
			{
				\foreach \j in {0,1}
				{
					\filldraw[fill=white] (\i+\j*0.707,\i-\j*0.707) circle (1.5pt);
				}
			}
			\node at (0.2,-0.507) {$x_K$};
			\node at (1,1) {$K$};
			\node at (1+0.707,1-0.707) {$K'(\cong K)$};
		\end{tikzpicture}
		\caption{the expansion $L \boxplus K$}
	\end{subfigure}
	\caption{The finite distributive lattices $L$ and $L\boxplus K$}\label{fig:lplk}
\end{figure}
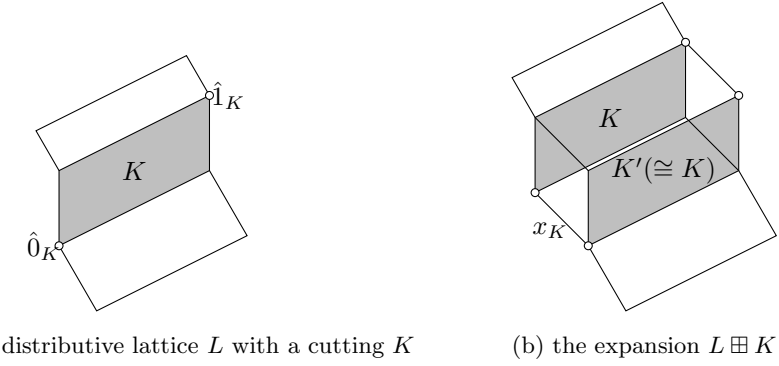

In particular, if $P-x \cong P*x$, then $\mathcal{F}(P) \cong \mathcal{F}(P-x) \mathbin{\square} \mathbf{2}$. Note that Theorem~\ref{th:cetfdl} holds for all $x\in P$, thus we get a more general method for drawing the Hasse diagram of $\mathcal{F}(P)$ than Stanley's in \cite[P293]{bStanl2011}. On the other hand, we have a consequence on Day's doubling construction.

\begin{remark}
	Let $L$ be a finite distributive lattice and let $K$ be a cutting of $L$. If $\hat{1}_K=\hat{1}_L$, then $L\uparrow \hat{0}_K$ in \cite{aErneHR2002} is the same as $L\boxplus K$. Thus, the operation $L\boxplus K$ is a special case of doubling construction on lattices in \cite{aDay1970}, and a generalization of the case in \cite[Section~2]{aErneHR2002}.
\end{remark}

In addition, we have two more general results than those in \cite[Section~2]{aErneHR2002}, namely Corollaries~\ref{th:ce} and \ref{th:mce}.

\begin{corollary}\label{th:ce}
	Let $P = P_0 \cup P_m$ be a finite poset, where $P_0 \cap P_m = \emptyset$ and $P_i=\{x_1,x_2,\dots,x_i\}$ for $i\in \{1,2,\dots,m\}$, then
	\begin{multline*}
	\mathcal{F}(P) \cong (\cdots((\mathcal{F}(P_0) \boxplus \mathcal{F}((P_0\cup P_1)*x_1)) \boxplus \mathcal{F}((P_0\cup P_2)*x_2)) \boxplus \cdots \\
	\boxplus \mathcal{F}((P_0\cup P_{m-1})*x_{m-1})) \boxplus \mathcal{F}((P_0\cup P_m)*x_m).
	\end{multline*}
\end{corollary}

Specially, if $P_0 = \emptyset$ in Corollary~\ref{th:ce}, a more general result is obtained by Theorem~\ref{th:ftfdl}.
\begin{corollary}\label{th:mce}
	Every finite distributive lattice with more than one element can be generated from one element lattice $\mathbf{1}$ by finite expansions $\boxplus$.
\end{corollary}

The consequences for cut element is also obtained.

\begin{corollary}\label{coro:cut}
	Let $P$ be a poset and $x\in P$.
	Then $\mathcal{F}(P-x)$ has a cut element, if and only if there exist two cut elements $M_1,M_2\in \mathcal{F}(P)$ such that $M_1 \cong M_2\cup \{x\}$, if and only if $P*x = \emptyset$.
\end{corollary}

Let $P$ and $S$ be posets, by 
$P+S$ and $P\dotplus S$ denote 
ordinal sum and vertical sum \cite{bDaveyP2002} of $P$ and $S$, respectively. As a consequence of Theorem \ref{th:cetfdl}, we have the following result.
\begin{corollary}[\cite{bDaveyP2002}]
	Let $P$ and $S$ be posets, then
	$\mathcal{F}(P+\mathbf{1})\cong\mathcal{F}(P)+\mathbf{1}, \mathcal{F}(\mathbf{1}+P)\cong\mathbf{1}+\mathcal{F}(P)$.
	Moreover,
	$\mathcal{F}(P\mathbin{\dot{\cup}} S) \cong \mathcal{F}(P)\mathbin{\square} \mathcal{F}(S)$ and $\mathcal{F}(P+S)\cong \mathcal{F}(P)\dotplus\mathcal{F}(S)$.
	
\end{corollary}

\section{Enumerative properties}
It is easy to show that the enumeration result on $L \boxplus K$ by counting. 
Let $R_L(K,x)$ denote the rank generating function of $K$ in $L$, and $R(L,x) := R_L(L,x)$; and let $h_L(x)$ denote the height of $x$ in $L$ for $x \in L$.
The rank generating function of $L \boxplus K$ is obtained by Theorem~\ref{th:cetfdl}.

\begin{proposition}
	Let $L$ be a finite distributive lattice and $K$ is a cutting of $L$. The rank generating function of $L\boxplus K$ is
	\[
	R(L\boxplus K,x) = R_L({\downarrow K},x) + xR_L({\uparrow K},x) = R({\downarrow K},x) + x^{h_L(\hat0_K)+1}R({\uparrow K},x)
	\]
\end{proposition}

In particular, we have the follow corollary.
\begin{corollary}\label{coro:rec-rgf}
	Let $L$ be a finite distributive lattice and $K$ is a cutting of $L$. The rank generating function of $L\boxplus K$ is
	\[
	R(L\boxplus K,x) = 
	\begin{cases}
	R(L,x) + x^{h_L(\hat0_K)+1}R(K,x), & \text{if } \hat1_K = \hat1_L; \\
	R(K,x) + xR(L,x), & \text{if } \hat0_K = \hat0_L.
	\end{cases}
	\]
\end{corollary}

Let $q_k(L)$ denote the number of convex Boolean lattice $\mathbf{B}_k$ in $L$, and let $q_k(L)=0$ if no $\mathbf{B}_k$ exists.
In fact, $q_k(L)$ is equal to the number of antichains of $k$ elements in $\operatorname{Mi}(L)$.

\begin{theorem}\label{th:enum}
	Let $L$ be a finite distributive lattice and let $K$ be a cutting of $L$. For $k\ge 0$, we have
	\[
	q_k(L\boxplus K) = q_k(L) + q_k(K) + q_{k-1}(K).
	\]
	Especially, $|L \boxplus K| = |L| + |K|.$
\end{theorem}

For some filter lattices of special posets, not only the number of elements is calculated, but also the structure is obtained.
For example, the undirected Hasse diagrams of filter lattices of the fences \cite{bStanl2011} and crown \cite{bBakerFR1970,aKellyR1974} are Fibonacci cubes and Lucas cubes, respectively \cite{pYao2007}.

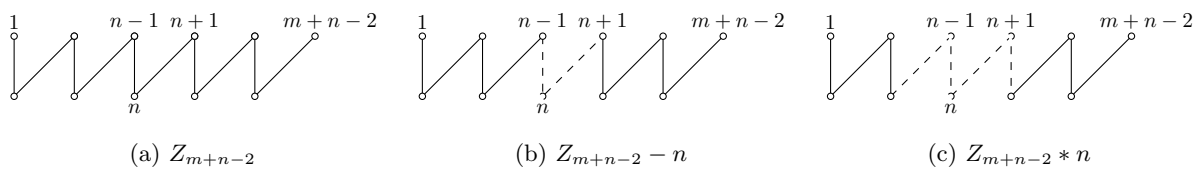
\begin{figure}[!htb]  
	\centering
	\begin{subfigure}[b]{0.32\linewidth}
		\centering
		\begin{tikzpicture}[scale=0.8, transform shape]
			\foreach \i in {0,...,4}
			{
				\draw (\i,1) -- (\i,0) -- (\i+1,1);
				\filldraw[fill=white] (\i,1) circle (1.5pt)
				(\i,0) circle (1.5pt)
				(\i+1,1) circle (1.5pt);
			}
			\node[above] at (0,1) {$1$};
			\node[above] at (2,1) {$n-1$};
			\node[above] at (3,1) {$n+1$};
			\node[below] at (2,0) {$n$};
			\node[above] at (5.25,1) {$m+n-2$};
		\end{tikzpicture}
		\caption{$Z_{m+n-2}$}
	\end{subfigure}
	\begin{subfigure}[b]{0.32\linewidth}
		\centering
		\begin{tikzpicture}[scale=0.8, transform shape]
			\draw[dashed] (2,1) -- (2,0) -- (3,1);
			\foreach \i in {0,1,3,4}
			{
				\draw (\i,1) -- (\i,0) -- (\i+1,1);
				\filldraw[fill=white] (\i,1) circle (1.5pt)
				(\i,0) circle (1.5pt);
			}
			
			\filldraw[fill=white,dashed] (2,0) circle (1.5pt);
			\filldraw[fill=white] (2,1) circle (1.5pt) (5,1) circle (1.5pt);
			\node[above] at (0,1) {$1$};
			\node[above] at (2,1) {$n-1$};
			\node[above] at (3,1) {$n+1$};
			\node[below] at (2,0) {$n$};
			\node[above] at (5.25,1) {$m+n-2$};
		\end{tikzpicture}
		\caption{$Z_{m+n-2}-n$}
	\end{subfigure}
	\begin{subfigure}[b]{0.32\linewidth}
		\centering
		\begin{tikzpicture}[scale=0.8, transform shape]
			\draw (1,1) -- (1,0) (3,0) -- (4,1);
			\foreach \i in {0,4}
			{
				\draw (\i,1) -- (\i,0) -- (\i+1,1);
			}
			\foreach \i in {2,3}
			{
				\draw[dashed] (\i-1,0) -- (\i,1) -- (\i,0);
			}
			\foreach \i in {0,1}
			{
				\foreach \j in {0,1}
				{
					\filldraw[fill=white] (\i,\j) circle (1.5pt)
					(\i+\j+3,\j) circle (1.5pt);
				}
			}
			
			\filldraw[fill=white,dashed] (2,0) circle (1.5pt) (2,1) circle (1.5pt) (3,1) circle (1.5pt);
			\node[above] at (0,1) {$1$};
			\node[above] at (2,1) {$n-1$};
			\node[above] at (3,1) {$n+1$};
			\node[below] at (2,0) {$n$};
			\node[above] at (5.25,1) {$m+n-2$};
		\end{tikzpicture}
		\caption{$Z_{m+n-2}*n$}
	\end{subfigure}
	\caption{An example for applications of Theorem~\ref{th:enum}}\label{fig:fibnum}
\end{figure}

By Theorems~\ref{th:cetfdl} and \ref{th:enum}, we have some results on Fibonacci cubes and Lucas cubes, that is Corollaries~\ref{coro:fib} and \ref{coro:lucas}.
\begin{corollary}\label{coro:fib}
	For the fences as shown in Figure~\ref{fig:fibnum}, we have
	\[
	\Gamma_{m+n-2} \cong (\Gamma_{m-2} \mathbin{\square} \Gamma_{n-1}) \boxplus (\Gamma_{m-3} \mathbin{\square} \Gamma_{n-2}),
	\]
	and
	\begin{equation}\label{eq:fc}
	q_k(\Gamma_{m+n-2}) = q_k(\Gamma_{m-2} \mathbin{\square} \Gamma_{n-1}) 
	+ q_k(\Gamma_{m-3} \mathbin{\square} \Gamma_{n-2}) + q_{k-1}(\Gamma_{m-3} \mathbin{\square} \Gamma_{n-2}),
	\end{equation}
	where $\Gamma_n$ is the $n$-th Fibonacci cube.
\end{corollary}
Hence some formulas for Fibonacci cubes are shown easily. For instance, let $k=0$ in \eqref{eq:fc}, we have
\[
F_{m+n} = F_mF_{n+1}+F_{m-1}F_n.
\]

In particular, let $m=2$, the equation \eqref{eq:fc} implies
\[
F_{n+2} = F_{n+1}+F_n.
\]


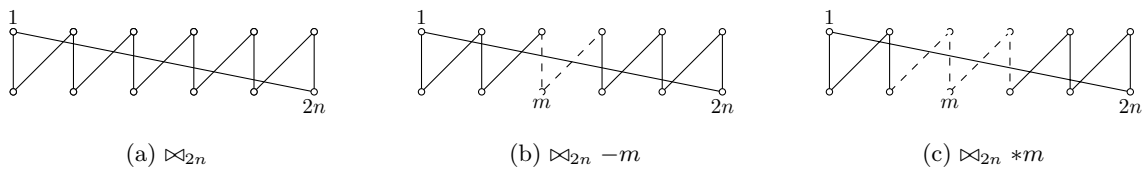
\begin{figure}[!htb]  
	\centering
	\begin{subfigure}[b]{0.32\linewidth}
		\centering
		\begin{tikzpicture}[scale=0.8, transform shape]
			\foreach \i in {0,...,4}
			{
				\draw (\i,1) -- (\i,0) -- (\i+1,1);
				\filldraw[fill=white] (\i,1) circle (1.5pt)
				(\i,0) circle (1.5pt)
				(\i+1,1) circle (1.5pt);
			}
			\draw (5,1) -- (5,0) -- (0,1);
			\foreach \i in {0,...,5}
			{
				\filldraw[fill=white] (\i,1) circle (1.5pt)
				(\i,0) circle (1.5pt);
			}
			
			\node[above] at (0,1) {$1$};
			\node[below] at (5,0) {$2n$};
		\end{tikzpicture}
		\caption{$\bowtie_{2n}$}
	\end{subfigure}
	\begin{subfigure}[b]{0.32\linewidth}
		\centering
		\begin{tikzpicture}[scale=0.8, transform shape]
			\draw[dashed] (2,1) -- (2,0) -- (3,1);
			\draw (5,1) -- (5,0) -- (0,1);
			\foreach \i in {0,1,3,4}
			{
				\draw (\i,1) -- (\i,0) -- (\i+1,1);
				\filldraw[fill=white] (\i,1) circle (1.5pt)
				(\i,0) circle (1.5pt);
			}
			\filldraw[fill=white] (5,0) circle (1.5pt);
			
			\filldraw[fill=white,dashed] (2,0) circle (1.5pt);
			\filldraw[fill=white] (2,1) circle (1.5pt) (5,1) circle (1.5pt);
			\node[above] at (0,1) {$1$};
			\node[below] at (5,0) {$2n$};
			\node[below] at (2,0) {$m$};
		\end{tikzpicture}
		\caption{$\bowtie_{2n}-m$}
	\end{subfigure}
	\begin{subfigure}[b]{0.32\linewidth}
		\centering
		\begin{tikzpicture}[scale=0.8, transform shape]
			\draw (3,0) -- (4,1) -- (4,0) -- (5,1) -- (5,0) -- (0,1) -- (0,0) -- (1,1) -- (1,0);
			\foreach \i in {2,3}
			{
				\draw[dashed] (\i-1,0) -- (\i,1) -- (\i,0);
			}
			\foreach \i in {0,1}
			{
				\foreach \j in {0,1}
				{
					\filldraw[fill=white] (\i,\j) circle (1.5pt)
					(\i+4,\j) circle (1.5pt);
				}
			}
			\filldraw[fill=white] (3,0) circle (1.5pt);
			
			\filldraw[fill=white,dashed] (2,0) circle (1.5pt) (2,1) circle (1.5pt) (3,1) circle (1.5pt);
			\node[above] at (0,1) {$1$};
			\node[below] at (5,0) {$2n$};
			\node[below] at (2,0) {$m$};
		\end{tikzpicture}
		\caption{$\bowtie_{2n}*m$}
	\end{subfigure}
	\caption{Another example for applications of Theorem~\ref{th:enum}}\label{fig:nmlc}
\end{figure}

\begin{corollary}\label{coro:lucas}
	For the crown as shown in Figure~\ref{fig:nmlc}, we have
	\[
	\Lambda_{2n} \cong \Gamma_{2n-1} \boxplus \Gamma_{2n-3}^*,
	\]
	and
	\begin{equation}\label{eq:nmlc}
	q_k(\Lambda_{2n}) = q_k(\Gamma_{2n-1}) + q_k(\Gamma_{2n-3}) + q_{k-1}(\Gamma_{2n-3}),
	\end{equation}
	where $\Lambda_n$ is the $n$-th Lucas cube.
\end{corollary}
Hence some formulas for Lucas cubes are shown easily too. For instance, using the same method in \eqref{eq:fc} for \eqref{eq:nmlc}, we have
\[
L_{2n} = F_{2n+1}+F_{2n-1}.
\]

Moreover, as similar to results in \cite{aMunarS2002b}, we can obtain the recurrence relations of rank generating functions of $\Gamma_n$ and $\Lambda_{2n}$, by Corollary~\ref{coro:rec-rgf}, respectively, but we shall omit them.

Let $d_k^-(L) = |\{\,x \in L \mid \operatorname{deg}_L^-(x) = k\,\}|$, 
$d_k^+(L) = |\{\,x \in L \mid \operatorname{deg}_L^+(x) = k\,\}|$, 
and
$d_k(L) = |\{\,x \in L \mid \operatorname{deg}_L(x) = k\,\}|$, where $\operatorname{deg}_L^-(x) = |\{\,y \in L \mid x \prec y\,\}|$, $\operatorname{deg}_L^+(x) = |\{\,y \in L \mid y \prec x\,\}|$ and $\operatorname{deg}_L(x) = \operatorname{deg}_L^-(x) + \operatorname{deg}_L^+(x)$ for $x \in L$.
It is not difficult to see that $d_0^-(L) =1$ and $d_1^-(L) = |\operatorname{Mi}(L)|$ for all finite distributive lattice $L$.

\begin{remark}
	If $L$ is a finite distributive lattice, thus both $d_k^-(L)$ and $d_k^+(L)$ is equal to the number of maximal antichains in $\operatorname{Mi}(L)$ with only $k$ elements.
	Therefore, it suffices to consider $d_k^-(L)$.
\end{remark}

In addition, the recurrence relations on $d_k^-(L)$ and $d_k(L)$ are obtained easily (see Figures~\ref{fig:deg} and \ref{fig:indeg}).
\begin{theorem}
	Let $L$ be a finite distributive lattice. If $K$ is a cutting of $L$, then
	\[
	d_k((L\boxplus K) \boxplus K) = d_k(L \boxplus K) + d_{k-2}(K).
	\]
	and
	\[
	d_k^-(L\boxplus K) = d_k^-(L) + d_{k-1}^-(K),
	\]
\end{theorem}

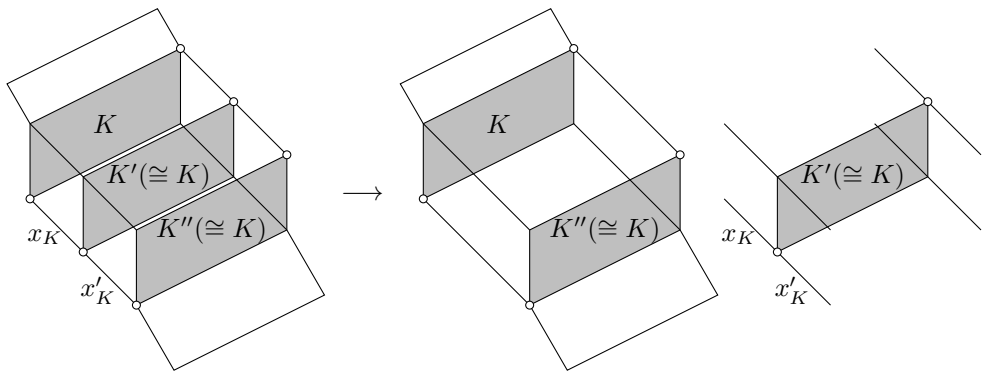
\begin{figure}[!ht]  
	\centering
		\begin{tikzpicture}[scale=1]
			\filldraw[fill=lightgray] (0,0) -- (2,1) -- (2,2) -- (0,1) -- cycle;
			\filldraw[fill=lightgray, xshift=0.707cm, yshift=-0.707cm] (0,0) -- (2,1) -- (2,2) -- (0,1) -- cycle;
			\filldraw[fill=lightgray, xshift=1.414cm, yshift=-1.414cm] (0,0) -- (2,1) -- (2,2) -- (0,1) -- cycle;
			\draw (0,1) -- (-1*0.306,1.732*0.306+1) -- (-1*0.306+2,1.732*0.306+2) -- (2,2);
			\draw[xshift=1.414cm,yshift=-1.414cm] (0,0) -- (1*0.5,-1.732*0.5) -- (1*0.5+2,-1.732*0.5+1) -- (2,1);
			\foreach \i in {0,1}
			{
				\foreach \j in {0,1}
				{
					\draw (2*\i,\i+\j)++(1.414,-1.414) -- (2*\i,\i+\j);
				}
			}
			\foreach \i in {0,2}
			{
				\foreach \j in {0,1,2}
				{
					\filldraw[fill=white] (\i+\j*0.707,\i-\j*0.707) circle (1.5pt);
				}
			}
			\node at (0.2,-0.507) {$x_K$};
			\node at (0.2+0.707,-0.507-0.707) {$x'_K$};
			\node at (1,1) {$K$};
			\node at (1+0.707,1-0.707) {$K'(\cong K)$};
			\node at (1+1.414,1-1.414) {$K''(\cong K)$};
		\end{tikzpicture}
	\raisebox{15ex}{${}\longrightarrow{}$}
		\begin{tikzpicture}[scale=1]
			\filldraw[fill=lightgray] (0,0) -- (2,1) -- (2,2) -- (0,1) -- cycle;
			\filldraw[fill=lightgray, xshift=4.707cm, yshift=-0.707cm] (0,0) -- (2,1) -- (2,2) -- (0,1) -- cycle;
			\filldraw[fill=lightgray, xshift=1.414cm, yshift=-1.414cm] (0,0) -- (2,1) -- (2,2) -- (0,1) -- cycle;
			\draw (0,1) -- (-1*0.306,1.732*0.306+1) -- (-1*0.306+2,1.732*0.306+2) -- (2,2);
			\draw[xshift=1.414cm,yshift=-1.414cm] (0,0) -- (1*0.5,-1.732*0.5) -- (1*0.5+2,-1.732*0.5+1) -- (2,1);
			\foreach \i in {0,1}
			{
				\foreach \j in {0,1}
				{
					\draw (2*\i,\i+\j)++(1.414,-1.414) -- (2*\i,\i+\j);
				}
			}
			\foreach \i in {0,1}
			{
				\foreach \j in {0,1}
				{
					\draw[xshift=4cm] (2*\i,\i+\j)++(1.414,-1.414) -- (2*\i,\i+\j);
				}
			}
			\foreach \i in {0,2}
			{
				\foreach \j in {0,2} 
				{
					\filldraw[fill=white] (\i+\j*0.707,\i-\j*0.707) circle (1.5pt);
				}
				\filldraw[fill=white] (\i+4.707,\i-0.707) circle (1.5pt);
			}
			\node at (0.2+4,-0.507) {$x_K$};
			\node at (0.2+4+0.707,-0.507-0.707) {$x'_K$};
			\node at (1,1) {$K$};
			\node at (1+4+0.707,1-0.707) {$K'(\cong K)$};
			\node at (1+1.414,1-1.414) {$K''(\cong K)$};
		\end{tikzpicture}
	\caption{Illustrating the recurrence relation of $d_k(L \boxplus K \boxplus K)$}\label{fig:deg}
\end{figure}

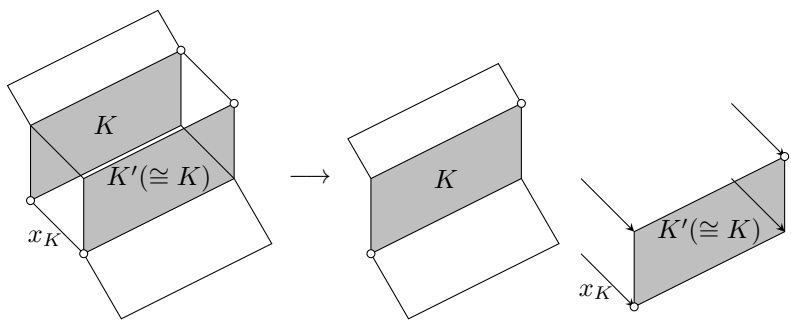
\begin{figure}[!ht]  
	\centering
		\begin{tikzpicture}[scale=1]
			\filldraw[fill=lightgray] (0,0) -- (2,1) -- (2,2) -- (0,1) -- cycle
			(0.707,-0.707) -- (2.707,1-0.707) -- (2.707,2-0.707) -- (0.707,1-0.707) -- cycle;
			\draw (0,1) -- (-1*0.306,1.732*0.306+1) -- (-1*0.306+2,1.732*0.306+2) -- (2,2);
			
			\draw[xshift=0.707cm,yshift=-0.707cm] (0,0) -- (1*0.5,-1.732*0.5) -- (1*0.5+2,-1.732*0.5+1) -- (2,1);
			
			\foreach \i in {0,1}
			{
				\foreach \j in {0,1}
				{
					\draw (2*\i,\i+\j)++(0.707,-0.707) -- (2*\i,\i+\j);
				}
			}
			\foreach \i in {0,2}
			{
				\foreach \j in {0,1}
				{
					\filldraw[fill=white] (\i+\j*0.707,\i-\j*0.707) circle (1.5pt);
				}
			}
			\node at (0.2,-0.507) {$x_K$};
			\node at (1,1) {$K$};
			\node at (1+0.707,1-0.707) {$K'(\cong K)$};
		\end{tikzpicture}
	\raisebox{12ex}{${}\longrightarrow{}$}
		\begin{tikzpicture}[scale=1]
			\filldraw[fill=lightgray] (0,0) -- (2,1) -- (2,2) -- (0,1) -- cycle;
			\filldraw[fill=lightgray, xshift=3.5cm, yshift=-0.707cm] (0,0) -- (2,1) -- (2,2) -- (0,1) -- cycle;
			\draw (0,1) -- (-1*0.306,1.732*0.306+1) -- (-1*0.306+2,1.732*0.306+2) -- (2,2);
			
			\draw (0,0) -- (1*0.5,-1.732*0.5) -- (1*0.5+2,-1.732*0.5+1) -- (2,1);
			
			\foreach \i in {0,1}
			{
				\foreach \j in {0,1}
				{
					\draw[xshift=3.5cm, yshift=-0.707cm, -stealth] (2*\i,\i+\j)++(-0.707,0.707) -- (2*\i,\i+\j);
				}
			}
			\foreach \i in {0,2}
			{
				\foreach \j in {0,1}
				{
					\filldraw[fill=white] (\i+\j*3.5,\i-\j*0.707) circle (1.5pt);
				}
			}
			\node at (3,-0.507) {$x_K$};
			\node at (1,1) {$K$};
			\node at (1+3.5,1-0.707) {$K'(\cong K)$};
		\end{tikzpicture}
	\caption{Illustrating the recurrence relation of $d_k(L \boxplus K)$}\label{fig:indeg}
\end{figure}

By definitions of $q_k(L)$ and $d_k^-(L)$, we have the results on them.
\begin{proposition}\label{prop:rel-qd-}
	Let $L$ be a finite distributive lattice and $m$ is the size of set of all maximal antichains in $\operatorname{Mi}(L)$. Then
	\[
	q_k(L) = \sum_{j=k}^m \binom jk d_j^-(L);
	\]
	moreover, by binomial transform,
	\[
	d_k^-(L) = \sum_{j=k}^m (-1)^{j-k}\binom jk q_j(L).
	\]
\end{proposition}

Let $Q_L(x) = \sum_{k=0}^m q_k(L) x^k$ and $D_L^-(x) = \sum_{k=0}^m d_k^-(L) x^k$, from Proposition~\ref{prop:rel-qd-}, a simple calculation leads to the following two corollaries.
\begin{corollary}
	Let $L$ be a finite distributive lattice and let $m$ denote the size of set of all maximal antichains in $\operatorname{Mi}(L)$.  We have
	\[
	Q_L(x) = D_L^-(1+x).
	\]
\end{corollary}

In addition, for a finite distributive lattice $L$, because of the fact that there is only one element covered by no element in $L$ and the element having exactly one cover is meet-irreducible, the following corollary is immediate.
\begin{corollary}
	Let $L$ be a finite distributive lattice. Then
	\[
	Q_L(-1) = q_0(L) - q_1(L) + q_2(L) - q_3(L) +\cdots = D_L^-(0) = 1,
	\]
	and
	\[
	\left. \frac{\mathrm{d}\, Q_L(x)}{\mathrm{d}\, x} \right|_{x=-1} = \left. \frac{\mathrm{d}\, D_L^-(x)}{\mathrm{d}\, x} \right|_{x=0} = d_1^{-}(L) = |\operatorname{Mi}(L)|.
	\]
\end{corollary}

This means that we verify the generalized Euler formula for polyhedrons with Euler characteristic $\chi = 1$ \cite{bDodsoP1997} and a special case of \cite{aSkrek2001}.

We finally remark that all results on filter can be written dually the conclusions on ideal and remain correct.

\section*{Acknowledgments}
\addcontentsline{toc}{section}{Acknowledgments}

The authors are grateful to the referees for their careful reading and many valuable suggestions.


\end{document}